\documentclass[11pt]{article}
\usepackage[margin=1.2in,letterpaper]{geometry}     
\usepackage{amsfonts}
\usepackage{amsmath,latexsym,amssymb,amsthm,enumerate,amsmath,amscd}
\usepackage[mathscr]{eucal}
\usepackage{bbm,color}
\usepackage{comment}
\usepackage{tikz-cd,rotating}
\usepackage{pdflscape}
\usepackage{authblk}

\definecolor{med-gray}{gray}{0.5}
\definecolor{gray1}{gray}{0.85}
\definecolor{gray2}{gray}{0.90}
\definecolor{gray3}{gray}{0.64}
\definecolor{gray4}{gray}{0.55}
\definecolor{verylight-yellow}{rgb}{1,1,0.7}
\definecolor{yellow}{rgb}{1,1,0.2}
\definecolor{vivid-blue}{rgb}{0.2,0,1}
\definecolor{light-pink}{rgb}{1,0.8,1}
\definecolor{med-pink}{rgb}{1,0.6,1}
\definecolor{aqua}{rgb}{0.0, 1.0, 1.0}
\definecolor{light-gray}{rgb}{0.5, 0.9, 0.5}

\usepackage{xy}
\usepackage{comment}
\usepackage{array}
\usepackage[colorlinks=true, linkcolor=blue, citecolor=blue, urlcolor=blue]{hyperref}

\theoremstyle{plain}
\newtheorem*{main}{Main Theorem}
\newtheorem{theorem}{Theorem}[section]
\newtheorem{proposition}[theorem]{Proposition}
\newtheorem{corollary}[theorem]{Corollary}
\newtheorem{lemma}[theorem]{Lemma}
\theoremstyle{definition}
\newtheorem{definition}[theorem]{Definition}

\newtheorem*{ack}{Acknowledgment}

\newtheorem{conjecture}[theorem]{Conjecture}

\newtheorem{remark}[theorem]{Remark}

\newtheorem{example}[theorem]{Example}

\newtheorem*{thm*}{Theorem}

\def\k{{\mathsf{k}}}
\def\cha{\mathrm{char}\ }
\def\Proj{\mathrm{Proj}\ }

\def\Pgl{\mathrm{PGL}}

\def\Hilb{\mathrm{Hilb}}

\def\Gor{\mathrm{Gor}}

\def\<{\left<}
\def\>{\right>}

\def\PP{\mathbb{P}}
\def\AA{\mathbb{A}}
\def\OO{\mathcal{O}}

\def\Hess{\mathrm{Hess}}

\def\Z{\mathfrak{Z}}

\def\max{\mathrm{max}}

\def\Hess{\mathrm{Hess}}

\def\m{\mathfrak{m}}

\def\Ann{\mathrm{Ann}}

\def\ch{\operatorname{char}}

\begin{document}
\date{}

\author[1]{Nancy Abdallah}
\author[2]{Nasrin Altafi}
\author[3]{Anthony Iarrobino}
\author[4]{Alexandra Seceleanu}
\author[5]{Joachim Yam\'{e}ogo}
\affil[1]{\small University of Bor\r{a}s, Bor\r{a}s, Sweden}
\affil[2]{Department of Mathematics, KTH Royal Institute of Technology, S-100 44 Stockholm, Sweden and Department of Mathematical Sciences, University of Copenhagen, 2100 Copenhagen, Denmark}
\affil[3]{Department of Mathematics, Northeastern University, Boston, MA 02115, USA}
\affil[4]{Department of Mathematics, University of Nebraska-Lincoln, Lincoln NE 68588, USA}
\affil[5]{Universit\'e C\^ote d'Azur, CNRS, LJAD, France.}

\title{Lefschetz properties of some codimension three Artinian Gorenstein algebras}
\renewcommand\footnotemark{}
\thanks{ \textbf{Email addresses:} {nancy.abdallah@hb.se, nasrinar@kth.se, a.iarrobino@northeastern.edu, aseceleanu@unl.edu, joachim.yameogo@unice.fr}}

\maketitle
\vspace{-2em}
\begin{abstract} 
Codimension two Artinian algebras $A$ have the strong and weak Lefschetz properties provided the characteristic is zero or greater than the socle degree. It is open to what extent such results might extend to codimension three AG algebras - the most promising results so far have concerned the weak Lefschetz property for such algebras.
We here show that every standard-graded codimension three Artinian Gorenstein algebra $A$ having low maximum value of the Hilbert function - at most six - has the strong Lefschetz property, provided that the characteristic  is zero.  When the characteristic  is greater than the socle degree of $A$, we show that $A$ is almost strong Lefschetz. This quite modest result is nevertheless arguably the most encompassing so far concerning the strong Lefschetz property for graded codimension three AG algebras.
\end{abstract}

\setcounter{tocdepth}{2}
\tableofcontents

\section{Introduction}
We consider Artinian Gorenstein (AG) algebras over a field $\sf k$. Recall that the Hilbert function of $A$ is the function $T(A):\mathbb{N}\to \mathbb{N}, T(A)_i=\dim_\k A_i$. The maximum value of the Hilbert function $T(A)$ of $A$ is called the {\em Sperner number} of $A$. The {\em socle degree} of such an algebra is $j=\max\{i \mid T(A)_i\neq 0\}$.

\begin{definition}\label{Lefdef} 
Consider a standard-graded  Artinian Gorenstein  algebra $A$ and an element $\ell\in A_1$. 
The pair $(A,\ell)$ is called \emph{weak Lefschetz} (WL) if  the maps $\ell:A_i\to A_{i+1}$ have full rank, that is, rank equal to $\min\{T(A)_i, T(A)_{i+1}\}$ for each integer $i$. 

The pair $(A,\ell)$ with $A$ of socle degree $j$ is called \emph{strong Lefschetz} (SL) if  the maps $\ell^{j-2i}:A_i\to A_{j-i}$ are bijections for each $i\in [0,  \lfloor j/2\rfloor]$.

We call $A$ \emph{weak Lefschetz} or \emph{strong Lefschetz}, respectively, if there is a linear form $\ell$ for which $(A,\ell)$ is WL, or SL, respectively. 
\end{definition}

The SL and WL properties can be defined equivalently in terms of the Jordan type $P_{A,\ell}$ of the map $\ell:A\to A$. The Jordan type is the partition of $\dim_{\sf k} A$ giving the Jordan block decomposition of this linear transformation.  The study of the Jordan type of Artinian algebras was initiated in the context of Lefschetz properties by T. Harima and J. Watanabe and collaborators \cite{Ha1,Ha2,HW,HW2,MW,HMNW,H-W}, since developed by many, then extended to more general Jordan types (see \cite{CGo,H-W,IMM}).

The pair $(A,\ell)$ is weak Lefschetz if and only if the number of parts of $P_{A,\ell}$ is the Sperner number of $A$; the pair is strong Lefschetz if and only if the partition $P_{A,\ell}$ is the conjugate $T(A)^\vee$ of the Hilbert function $T(A)$ - when $T(A)$ is reordered as a partition. For the equivalence of the definitions in the SL case see \cite[Prop. 3.64]{H-W} or \cite[Prop 2.10]{IMM}.  

It was known, as a consequence of J. Brian\c{c}on's work on standard bases, that
in codimension two all standard-graded or local Artinian algebras $\mathcal A$ (calligraphic $\mathcal A$ for local algebra, $A$ for graded) are strong Lefschetz, provided
${\sf k}$ has characteristic zero or characteristic $p>j$, the socle degree of $\mathcal A$ \cite[Proposition 2.15]{IMM}.
A main tool in studying codimension three AG algebras is the Pfaffian structure theorem of D. Buchsbaum, D. Eisenbud \cite{BuEi}. A consequence, realized after work by many, including \cite{Di,CoVa,Ha2} is that graded codimension three AG algebras of a given Hilbert function $T$ are parametrized by a smooth irreducible variety $\Gor(T)$ of known dimension  (see \cite[Chapter~5]{IK} for an account). We say that a family of codimension three AG algebras $A$ is \emph{general} of Hilbert function $T$ if it is parametrized by an (unspecified) open dense subset of $\Gor(T)$.

 Work of T. Harima, J. Migliore, U. Nagel and J. Watanabe \cite{HMNW}  showed that all codimension three graded {\it complete intersections} (CI) are weak Lefschetz when $\cha {\sf k}=0$: this result already uses the Grauert M\"{u}lich theorem for vector bundles on $\mathbb P^2$. The question of whether all codimension three AG algebras $A$ have the weak Lefschetz property - equivalently, whether for a general linear form $\ell$ the partition $P_{A,\ell}$ has the Sperner number of parts - has been reduced in arbitrary characteristic to the case of $A$  compressed Gorenstein of odd socle degree in \cite{BMMNZ}. Here, a \emph{compressed} Gorenstein algebra $A$ of codimension $r$ and socle degree $j$ is one having the maximum possible symmetric Hilbert function $T(A)_i=\min\{r_i,r_{j-i}\}$ where $R={\sf k}[x_1,\ldots,x_r]$ and $r_i=\dim_{\sf k}R_i$. 
 
 The graded compressed Artinian Gorenstein algebras of embedding dimension $r$ and socle degree $j$ form an irreducible family parametrized by an open dense subset of $\mathbb P(S_j)$, where $S={\k}[X_1,\ldots,X_r]$; the algebras $A$  in the family corresponds to their Macaulay dual generators $ F_A\in S_j $.   A \emph{general} compressed Gorenstein algebra of any codimension and socle degree was known to be strong Lefschetz  \cite[Cor.~3.40]{H-W}.  Also monomial AG algebras of  arbitrary codimension and socle degree (they must be CI) are strong Lefschetz (\cite{St} or \cite[Prop. 3.4, Cor. 3.5]{Wa2}). Finally, the second author has shown that for any codimension three Gorenstein sequence $T$ (one occurring for a graded AG algebra, see \S \ref{Gorsec}, Lemma \ref{cod3SIem}) a general element of $\Gor(T)$ is strong Lefschetz, a consequence of \cite[Theorem 3.2]{Al} (see Theorem~\ref{althm} below). The question of which codimension three (graded) Artinian Gorenstein algebras are strong Lefschetz was largely open, outside of those cases. 
 
 Our main result employs the following invariants: the \emph{order} $\nu$ of an Artinian algebra $A=R/I$ given by  $\nu=\min\{i \mid T(A)_i \neq \dim_{\sf k}(R_i)\}$ and the multiplicity $k$ of the Sperner number in the Hilbert function $T(A)$.
Additionally, we need the following definition.
\begin{definition}[Almost strong Lefschetz] We say that an AG algebra $A$ is \emph{almost strong Lefschetz} (almost SL) if all the maps $\ell^{j-2i}:A_i\to A_{j-i}$ for $i$ satisfying $2\le i\le  \lfloor j/2\rfloor$ or $i=0$ are isomorphisms for a general linear form $\ell$. That is, only the map $\ell^{j-2}:A_1\to A_{j-1}$ among those in Definition \ref{Lefdef} may not be an isomorphism if $A$ is almost SL.  
\end{definition}

We are now ready to state our main result.

  \begin{main}
  \label{thm:main}
 The following classes of codimension three standard-graded AG algebras $A$ over a field $\sf k$ of characteristic zero have the strong Lefschetz property. When  ${\sf k}$ is an infinite field of characteristic $p>j$, the socle degree of $A$,  then AG algebras in these classes are almost strong Lefschetz.
  
  \begin{enumerate}[i.]
    \item AG algebras with Sperner number less than or equal six, as described in table \eqref{table1}
   \item some AG algebras of order two, as described in table \eqref{table2}.
  \end{enumerate}
\end{main}
\begin{proof} The proof follows from the characterization of the codimension three Gorenstein sequences of Sperner number at most six or of order two in Corollary \ref{charcor}, and Theorems \ref{SLthm}, \ref{thm1,3,4}, and \ref{thm1,3,5,6}.
\end{proof}
   
 We list in the following two tables the Hilbert functions of AG algebras that are covered by our theorem, that is, the Hilbert functions of all codimension three AG algebras either having Sperner number at most six (Table \ref{table1}), or having order two and not beginning $H=(1,3,5,6,7,\ldots)$ or $H=(1,3,5,7,\ldots)$ (Table \ref{table2}). In the second column we list the pertinent Theorem.
\par
\begin{table}[h]
 \caption { SL property for AG quotients of ${\k}[x,y,z]$ of Sperner number $s\le 6$}\label{table1}
   \begin{equation}
       \begin{array}{|c|c|}
  \hline
  T=H(A), \, k\ge 1&\mathrm{SL}\,\cha \k=0\\\hline\hline
  (1,3^k,1)&\mathrm{Theorem\, \ref{SLthm}} \\\hline
  (1,3,4^k,3,1)&\mathrm{Theorem\,\ref{SLthm}}\\\hline
  (1,3,4,5^k,4,3,1)&\mathrm{Theorem\,\ref{thm1,3,4}}\\\hline
  (1,3,4,5,6^k,5,4,3,1)&\mathrm{Theorem\,\ref{thm1,3,4}}\\\hline
  (1,3,4,5^k,4,3,1)&\mathrm{Theorem\,\ref{thm1,3,4}}\\\hline
  (1,3,5^k,3,1)&\mathrm{Theorem \,\ref{SLthm}}\\\hline
  (1,3,5,6^k,5,3,1)&\mathrm{Theorem\, \ref{thm1,3,5,6}}\\\hline
  (1,3,6^k,3,1),&\mathrm{Theorem \,\ref{SLthm}}\\\hline
  \end{array}
  \end{equation}
  \end{table}
    
 \begin{table}[h]
\caption {SL property for AG quotients of ${\k}[x,y,z]$ of Order 2:}\label{table2}
   \begin{equation}
   \begin{array}{|c|c|}
  \hline
  T=H(A),\, k\ge 1&\mathrm{SL}\,\cha \k=0\\\hline\hline
  (1,3^k,1)&\mathrm{Theorem \,\ref{SLthm}}\\\hline
  (1,3,4^k,3,1)&\mathrm{Theorem \,\ref{SLthm}}\\\hline
   (1,3,5^k,3,1)&\mathrm{Theorem \,\ref{SLthm}}\\\hline
  (1,3,4,5,\ldots, (s-1),s^k,(s-1),\ldots,3,1)&\mathrm{Theorem\,\ref{thm1,3,4}}\\\hline
  (1,3,5,6^k,5,3,1)&\mathrm{Theorem\, \ref{thm1,3,5,6}}\\\hline
  (1,3,5,6,\ldots, (s-1),s^k,(s-1),\ldots, 3,1)&\mathrm{open}\, \ast\\\hline
  (1,3,5,7,\ldots, 2t+1,2t+2,2t+3\ldots, s^k,\ldots 3,1) &\mathrm{open}\,\ast\ast\\\hline
  \end{array}
  \end{equation}
  \end{table}
\vskip 0.2cm
   Main tools on which our results rely  include  a study of the relation between an AG ideal $I$ and $I:x$, which is also an AG ideal, the theory of Macaulay inverse systems, and the connection between Hessian matrices and the Lefschetz properties developed in \cite{MW,Go,GoZa}. Moreover, we employ the morphism from
   $\Gor(T)$ to the punctual Hilbert scheme $\Hilb^H(\mathbb P^2)$ of projective space, where $H$ is the first, non-decreasing portion of $T$, when $T$ has a consecutive subsequence $(s,s,s)$ (Theorem~\ref{pithm}), a consequence of the D. Buchsbaum-D.~Eisenbud Pfaffian structure theorem.  See Section~\ref{s:pi} for details regarding this morphism. 
  
  \section{Tools}

 \subsection{Macaulay inverse systems}
Let $R=\k[x_1,\ldots,x_r]$ be a polynomial ring. Next let $S=\k_{DP}[X_1,\ldots,X_r]$ be a divided power algebra on which $R$ acts by contraction, see e.g. \cite[Appendix A]{IK}: 
$$x_i\circ X_1^{a_1}\cdots X_i^{a_i}\cdots X_r^{a_r}=\begin{cases} X_1^{a_1}\cdots X_i^{a_i-1}\cdots X_r^{a_r} & \text{if} \ a_i>0\\ 0 & \text{if} \ a_i=0\\ \end{cases}.$$

It is well known by  work of  F.H.S. Macaulay \cite{Mac} that a homogeneous ideal $I\subset R$ is $\mathfrak{m}$-primary irreducible - defines an AG quotient algebra $A=R/I$ of socle degree $j$ if and only if there exists a homogeneous form $F\in S_j$ for which $I=\Ann(F)=\{r\in R \ | \ r\circ F=0\}$.  

\begin{definition}
\label{def:F}
Let $R={\sf k}[x_1,\ldots,x_r]$, and let $A=R/I$ be a graded AG algebra of socle degree $j$. A polynomial $F\in S_j$ satisfying $I=\Ann(F)$ is called a {\em Macaulay dual generator} of $A$.
\end{definition}
A Macaulay dual generator of $A$ is unique up to a non-zero constant multiple.
\begin{remark}
\label{rem:B}
  Suppose that $A=R/I$ is an AG algebra with $I=\Ann(F)$ and let $\omega\in A_{j-k}$ and  $J=(I\colon\omega)$. Then by \cite[Lemma 4]{Wa1} we have that $J=\Ann(G)$ where $G=\omega\circ F$. In particular $R/J$ is AG.
 \end{remark}
 
 \subsection{Hessians  and the strong Lefschetz property}
 \label{s:Hessians}
\begin{definition}[{\cite[Definition 3.1]{MW}}]
Let $F$ be a polynomial in $S$ and $A= R/\Ann(F)$ be its associated AG algebra. Let $\mathcal{B}_{i} = \lbrace \alpha^{(i)}_m\rbrace_m$ be a  $\mathsf{k}$-basis of $A_i$. The entries of the  $i$-th Hessian matrix of $F$ with respect to $\mathcal{B}_i$ are given  by
$$
(\Hess^i(F))_{u,v}=(\alpha^{(i)}_u\alpha^{(i)}_v \circ F).
$$
Up to  a non-zero constant multiple the determinant $\det \Hess^i(F)$ is independent of the basis $\mathcal{B}_i$. 
We note that when $i=1$ the form $\Hess^1(F)$ coincides with the usual Hessian.   
\end{definition} 
 T. Maeno and J. Watanabe provided a criterion for AG algebras to be SL and to identify the SL elements using the Hessians. See  \cite[\S 3.6]{H-W} for an account.

\begin{theorem} [{\cite[Theorem 3.1]{MW}}]
\label{thm:hessian}
Assume $\cha \k=0$. A linear form $\ell=a_1x_1+\cdots +a_nx_r\in A_1$ is a strong Lefschetz element of $A=R/\Ann(F)$ if and only if  
$$
(\det\Hess^i(F))(a_1,\dots ,a_n)\neq 0,
$$
for $i=0,1,\dots , \lfloor\frac{j}{2}\rfloor$. In particular, for $i=0,1 ,\dots , \lfloor\frac{j}{2}\rfloor$ the multiplication map $\times \ell^{j-2i} :A_i\longrightarrow A_{j-i}$ has maximal rank if and only if $(\det\Hess^i(F))(a_1,\dots ,a_n)\neq 0$. 
\end{theorem}
 
 The following theorem due to P.~Gordan and M.~Noether in \cite{GoNo} was reproved in \cite[Theorem 1.2]{Los}, \cite{WatdeB}, \cite{BFP}.
 
 \begin{theorem}[Gordan-Noether]\label{GNlem} If $F$ is a form expressed in at most four variables and the characteristic of $\sf k$ is zero, then the Hessian determinant $\det \Hess(F)$ is identically zero  if and only if $F$ is  a cone: that is, if $F$ is annihilated by a linear form (element of $R_1$) in the Macaulay duality.
\end{theorem}

  In particular, this implies by Theorem \ref{thm:hessian} that if $\cha \k=0$, setting $j=\deg(F)$, $I=\Ann(F)$ with $I\subset \m^2$, and $A=R/I$  for a polynomial ring $R$  of codimension at most four, the map $\ell^{j-2}:A_1\to A_{j-1}$ is  an isomorphism for a general linear form $\ell\in A_1$. It is this result, which we have not been able to extend to characteristic $p>j$, that led to our introducing the ``almost strong Lefschetz'' terminology.

\subsection{The morphism $\pi:\Gor(T)\to \Hilb^s(\mathbb P^2)$}
\label{s:pi}

The Gorenstein sequences of codimension three in which the Sperner number occurs at least three times play a distinguished role in our results. This is because of the following theorem - a connection between the AG algebras whose Hilbert function is such a sequence to certain 0-dimensional projective schemes in $\mathbb P^2$. Recall that an ideal $J\subset R$ is \emph{saturated} if $J:\m=J$.
We denote by $\underline{s}$ the infinite sequence $(s,s,\ldots)$.

\begin{theorem}[{\cite[Theorem 5.31]{IK}}]
\label{pithm}
Assume $\sf k$ is an infinite field \footnote{In \cite{IK} there is an implicit assumption that $\sf k$ is algebraically closed, but the result holds for $\sf k$ infinite.} and let $T$ be the Hilbert function of a codimension three Artinian Gorenstein algebra $A=R/I$ of socle degree $j$. Suppose  $T\supset (s,s,s)$, that is, $T$ contains a consecutive subsequence of at least three $s$ where $s$ is the Sperner number of $T$,  and let $\tau$ be the smallest integer such that $T(A)_\tau=s$.  Let $H=(T_{\le j/2},s, s, s, \underline{s})$. 
Then\begin{enumerate}[i.]
\item The ideal $J=(I_{\le \tau+1})\subset R$ is a saturated ideal defining a subscheme $\Z=\Proj(R/J)\subset \mathbb{P}^2$, having length $s$, Castelnuovo-Mumford regularity $\tau+1$, and Hilbert function $H$.
\item The scheme $\Z$ is the unique degree-$s$ scheme in $\mathbb P^2$ whose ideal is contained in $I$.
\item The quotient ideal $I/J$ defines the dualizing sheaf $A_\Z$ of $\Z$ \cite{Bo}.
\item The map $I\to J$ defines a morphism:  $\pi:\Gor(T)\to \Hilb^H(\mathbb P^2)\subset \Hilb^s(\mathbb P^2)$ whose image contains the open dense subscheme $\mathrm {Sm}^H(\mathbb P^2)$ parametrizing smooth length-$s$ punctual subschemes of $\mathbb P^2$ having Hilbert function $H$.
\item The minimal resolution of $J$ can be explicitly constructed from the  resolution of $I$.

\end{enumerate}
\end{theorem}
A similar result is valid when the Gorenstein sequence $T\supset (s,s)$ or just has Sperner number $s$, but is restricted to a proper scheme sublocus $\Gor_{sch}(T)\subset \Gor(T)$ \cite[Theorems~5.39,5.46]{IK}.
\begin{definition}
\label{def:tightann}
In the language of \cite[\S 5.1]{IK}, if the Macaulay dual generator for the AG algebra $A$ is $F$ (Definition \ref{def:F}) and if $A$ has Sperner number $s$, then a length-$s$ scheme $\Z\in \mathbb{P}^2$ such that $J=I(\Z)\subset I$ is termed a \emph{tight annihilating scheme} of $F$. 
\end{definition}
In particular the scheme $\Z$ in Theorem \ref{pithm}  is a tight annihilating scheme of $F$.
A tight annihilating scheme is unique under some further conditions \cite[Theorem 5.3]{IK}. The role of tight annihilating scheme here is due to the following Lemma.

\begin{lemma}
\label{tightlemma}
 Let $F\in {\sf k}[X_1,\ldots,X_r]_j$ be the Macaulay dual generator of the AG algebra $A={\sf k}[x_1,\ldots,x_r]/I, I=\Ann F$ of Hilbert function $T$. Assume that $F$ has a tight length-$s$ punctual annihilating scheme $\Z$, and suppose $\tau=\tau(\Z)=\min\{i\mid T_i=s\}$.  The dimension-one algebra $B={\sf k}[x_1,\ldots ,x_r]/I(\Z), I(\Z)=(I_\le {j/2})$ of Hilbert function $H=(T_{\le j/2},\underline{s})$ has a non-zero divisor $\ell\in A_1$. For each pair $(u,v)$ with $0\le u\le v\le j-\tau$ the multiplication homomorphism $\ell^{v-u}: A_u\to A_v$ is an injection - has maximal rank. In particular
for $\tau\le u\le j-\tau$ this map is an isomorphism, and $\ell^{j-2i}:A_i\to A_{j-i}$ is an isomorphism for 
$\tau\le i\le j/2$.
\end{lemma}

\subsection{Gorenstein sequences in codimension three}\label{Gorsec}
 Recall that an \emph{O-sequence} $T=(1,t_1,\ldots , t_j,0)$ is one that occurs as the Hilbert function of an Artinian algebra: it satisfies certain Macaulay conditions \cite[\S 4.2]{BH}. A \emph{Gorenstein sequence} $T=(1,t_1,\ldots, t_j=1,0)$ is a symmetric sequence that occurs as the Hilbert function of a standard-graded Artinian Gorenstein algebra.
 We need
 \begin{definition}[SI sequence]
\label{Tdef}  Let $T=(1,r,\ldots t_i,\ldots, r, 1_j)$ be a sequence symmetric about $j/2$.
 Let $j'=\lfloor j/2\rfloor$. We say that $T$ is an \emph{SI-sequence}, if, letting $\delta_i(T)=t_i-t_{i-1}$, we have
\begin{equation}\label{DTeq}
\Delta T=(1,\delta_1(T),\delta_2(T),\ldots, \delta_{j' }(T)) \text { is an O-sequence}.
\end{equation}
\end{definition}
 
 The Pfaffian structure theorem of D. Buchsbaum and D. Eisenbud showed that a codimension three Gorenstein ideal $I$ has a  length 3 minimal resolution where the middle map is an alternating matrix whose diagonal Pfaffians (square roots of the diagonal minors) are the generators of $I$ (\cite{BuEi}, also see \cite[Theorem B2, Appendix B.2]{IK} for the statement. 
 
  A consequence of the Pfaffian structure theorem for Artinian Gorenstein algebras $A$ of codimension three AG, is to characterize their Hilbert functions.
  \begin{lemma}\label{cod3SIem}(R. Stanley \cite[Theorem 4.2]{St}, see also \cite[Theorem 5.25i]{IK}) The symmetric sequence $T$ with $r=3$ is a Gorenstein sequence if and only if $T$ is a SI-sequence.
  \end{lemma}
  We denote by $\nu(T)$, the \emph{order} of $T$, the smallest $\nu$ such that $T_\nu\not=\dim_{\sf k}R_\nu$: this is the order $\nu(I)$ - lowest degree of a generator of any ideal $I$ defining a quotient $A=R/I$ of Hilbert function $T$.
 The necessary and sufficient condition for $\Delta T$ to be an O-sequence of an algebra of codimension two is
\par
\begin{equation}\label{DT2eq} \delta_i(T)=i+ 1\text{ for $i<\nu(T)$ and }\delta_i(T)\ge \delta_{i+1}(T) \text{ if } i\ge \nu(T).
\end{equation}

This allows us to determine the possible Gorenstein sequences for codimension three AG algebras having order two or Sperner number at most six. These are summarized in tables \eqref{table1}, \eqref{table2}.

\begin{corollary}
\label{charcor} 
The codimension three Gorenstein sequences of Sperner number at most six are
$T=(1,3^k,1),(1,3,4^k,3,1),(1,3,4,5^k,4,3,1), (1,3,4,5,6^k,5,4,3,1), (1,3,5^k,3,1)$, \par\noindent $(1,3,5,6^k,5,3,1) $ and $(1,3,6^k,3,1)$ where $k\ge 1$. The additional Gorenstein sequences of order two are $T=(1,3,4,5,\ldots, s-1,s^k,s-1,\ldots,3,1),   (1,3,5,6,\ldots, s-1,s^k,s-1,\ldots, 3,1)$, and
  $(1,3,5,7,\ldots, 2t+1,2t+2,2t+3\ldots, s^k,\ldots 3,1)$, where $s\ge 7$ and $k\ge 1$.
\end{corollary}
\begin{proof}
From \eqref{DT2eq} we obtain under the assumption that the Sperner number does not excede 6 that $\delta_i(T)=(1,2,\underline{0})$ or $\delta_i(T)=(1,2,1,\underline{0})$ or $\delta_i(T)=(1,2,1,1,\underline{0})$ or $\delta_i(T)=(1,2,1,1,1,\underline{0})$ or $\delta_i(T)=(1,2,2,\underline{0})$. This is an exhaustive list because the Sperner number is the sum of the $\delta_i(T)$ vector. This yields the first claim.

Under the assumption that the order is two the difference vector must have the form $\delta_i(T)=(1,2,\underline{0})$ or $\delta_i(T)=(1,2,1^{s-3},\underline{0})$ or $\delta_i(T)=(1,2^{t},1^{s-2t-1}, \underline{0})$. The first case was discussed above. The second  yields $T=(1,3,4,5,\ldots, s-1,s^k, s-1,\ldots,3,1)$. The third yields both of the remaining Hilbert functions according to whether $t=2$ or $t>2$.
\end{proof}
We now elaborate on structural aspects for one of the above mentioned Hilbert sequences.

\begin{lemma}
\label{lem**}
Assume that  $\sf k$ is of characteristic zero, or infinite of characteristic $p\geq 3$.
 
If a Gorenstein sequence has the form  
\begin{equation}
\label{specialeqn}
T = (1,3,5,7,\ldots, 2t+1,2t+2,2t+3,\ldots, s^k,\ldots 3,1_j),
\end{equation}
 with $\delta_i(T)=(1,2^{t},1^{s-2t-1}, \underline{0})$ for $t\geq 3, s\geq 2t+1$, then any AG algebra $A=\k[x,y,z]/I$ of Hilbert function $T$ satisfies $I_2\cong (xy)$ or $I_2\cong (z^2)$, under $\Pgl(3)$ action.
\end{lemma}
\begin{proof}
 Assume by way of contradiction that $I\cong (f,g,\ldots)$, with $f\in R_2$ irreducible, and $g$ the next minimal generator, of degree $a\geq 4$. Since $f$ is irreducible, the ideal $(f,g)$ is a complete intersection of  generator degrees $(2,a)$ and its Hilbert function $T=T(A)$ begins $T=(1,3,5, 7,\ldots, 2a-1,\underline{2a})$. This yields $a=t+1$ and a componentwise inequality 
 \[
 T=T(A)\le T(R/(f,g))=(1,3,5,\ldots, 2t+1,2t+2,2t+2,\ldots),
 \] contradicting the equation \eqref{specialeqn}. 
  \end{proof}
     
We say a symmetric unimodal sequence $T$ is a \emph{strong Lefschetz sequence} if there is at least one standard-graded Artinian Gorenstein algebra having Hilbert function $T$. The second author showed the following.
\begin{theorem}\label{althm}\cite[Theorem 3.2]{Al} Let $R={\sf k}[x_1,\ldots, x_r]$ and suppose $\cha {\sf k}=0$.\begin{enumerate}[(i).]
\item Every AG algebra quotient having socle degree $j\ge 2\tau(\Z)$ of the coordinate ring $A(\Z)$ of a smooth projective scheme $\Z$ in $\mathbb P^{r-1}$, is strong Lefschetz.
\item A Gorenstein sequence $T=H(R/I) $ is a strong Lefschetz sequence if and only if it is an SI-sequence.
\end{enumerate}
\end{theorem}
Her method of proof of (i) uses the Hessian criterion; the proof of (ii) uses a result of P. Maroscia that $\Hilb^H(\mathbb P^{r-1})$ where $\Delta(H)$ is an O-sequence contains a smooth scheme (see \cite[Thm 5.21]{IK}.  L.J. Billera and C. W. Lee \cite{BiL} and R. Stanley \cite{St2} showed in 1980 that weak Lefschetz Gorenstein sequences are SI when $\cha {\sf k}$ is zero. T. Harima generalized this weak Lefschetz statement to a field of arbitrary characteristic using an algebraic method - linkage - and a technique of J. Watanabe \cite[proof of Theorem 3.8]{Wa2} in 1995 \cite{Ha1}.\footnote{Part (ii) of Theorem \ref{althm} also follows from \cite{BiL} and \cite{St2}: the former shows the $g$-Theorem that an SI sequence is the $h$-vector of a simplicial polytope $P$;  the latter shows that $P$ defines an algebraic variety $X_P$ whose cohomology ring is strong Lefschetz and has Hilbert function $h$. \par Corollary 4.6 of \cite{HMNW} shows that over an infinite field
the set of unimodal Hilbert functions $T$ possible for weak Lefschetz Artinian algebras and the set possible for strong Lefschetz Artinian algebras are the same:  they are those such that the positive part of the first differences $\Delta T$ is an O-sequence \cite[Proposition~3.5]{HMNW}. This is different than identifying the sets of Hilbert functions that are possible for weak Lefschetz AG algebras and for strong Lefschetz AG algebras using Theorem \ref{althm} and the earlier results. The characterization of strong Lefschetz Gorenstein sequences remains open in characteristic $p$.}

\section{Lefschetz properties}

Recall that all Artinian algebras $A$ of codimension two are strong Lefschetz provided the characteristic of the base field is zero or larger than the socle degree of $A$.

For graded Gorenstein algebras of codimension three, some important cases of weak Lefschetz or strong Lefschetz are already known; the most relevant for this paper are contained in 
\cite{HMNW,BMMNZ}. Several have informally conjectured even that all graded codimension three Gorenstein algebras might have the strong Lefschetz property, a conjecture we regard as extraordinarily ambitious, given that perhaps the strongest result - that codimension three graded complete intersections are weak Lefschetz over $\cha {\sf k}=0$ uses the Grauert-M\"{u}lich theorem for vector bundles on $\mathbb P^2$ \cite{HMNW}.
 The main pertinent results known before 2017 can be summarized as follows:
 
\begin{proposition}
\label{WLlem} 
   \begin{enumerate}[i.]
   \item If $\cha{\sf k}=0$,  all codimension three complete intersections are weak Lefschetz (WL)  (\cite{HMNW}, see also \cite[Theorem 3.48, Lemma~3.49]{H-W} for an exposition)
   \item If $\cha{\sf k}=0$,  all codimension three AG algebras of socle degrees $2e-1$ and $2e$ over an infinite field $\sf k$ are WL if and only if all compressed codimension three AG algebras of odd socle degree $2e-1$ are WL  \cite[Corollary 2.5]{BMMNZ}.
   \item If $\cha {\sf k}=0$ or $\cha {\sf k}=p>3$ and $\sf k$ is an infinite field then all codimension three compressed AG algebras of socle degree $3$ and $5$ are WL (for $j=5$ see \cite[Theorem~3.8]{BMMNZ} (where $H=(1,3,6,6,3,1))$.\footnote{They also show that when $j=5$ and $\cha \sf k=3$ then the only non-WL exception is when $I= (x^2y,x^2z,y^3,z^3, x^4 + y^2z^2)$, after a change of variables. The case $j=3$ was known in $\cha {\sf k}=0$ and is worked out below when $\cha {\sf k}=p$ (Lemma \ref{13k1SLlem}).}  \footnote{A. Vraciu writes in her MathSciNet review of \cite{BMMNZ}: ``The proof involves subtle geometric properties of certain classical configurations of known points and Hesse configurations [in $\mathbb P^2$ ].''} 
   \item  If $\cha \sf k=0$, all codimension three AG algebras of socle degree  $j\leq 5$ are  SL  \cite[Corollary 3.13]{BMMNZ}.\footnote{The characteristic $\sf k=0$ assumption comes from use of the Gordan-Noether theorem to show the map $\ell^{j-2}: A_1\to A_{j-1}$ has rank 3.}
\item If $\cha{\sf k}=0$, all codimension three algebras defined by powers of linear forms are WL \cite{ScSe}.
    
     \end{enumerate}
\end{proposition}
In addition, as a consequence of Theorem \ref{althm} we have
\begin{proposition}\label{gencompprop} Let $T$ be a codimension three Gorenstein sequence and assume $\cha {\sf k}=0$. Then an open dense subset of $\Gor(T)$ is comprised of strong Lefschetz algebras.
\end{proposition}
\begin{proof} By Lemma \ref{cod3SIem} $T$ is an SI sequence. By Theorem \ref{althm} there is a strong Lefschetz AG algebra of Hilbert function $T$. Since in codimension three $\Gor(T)$ is irreducible \cite{Di}, and since being strong Lefschetz is an open condition on a family of algebras having constant Hilbert function - it corresponds to maximal rank of the matrices for multiplication by $\ell^{j-2i}: A_i\to A_{j-i}$ - the conclusion follows.
\end{proof}

   In the following we begin a proof of our main result. The proof proceeds by case analysis bases on the possible Hilbert functions listed in   tables \eqref{table1} and \eqref{table2}.
   
   We will assume throughout that ${\sf k}$ is an infinite field of characteristic zero, or characterstic $p> j$, where $j$ is the socle degree of the AG algebras being considered. We discuss $\cha {\sf k}=p>j$ explicitly  in Section \ref{charpsec}.
 Unless otherwise specified we take $R={\sf k}[x,y,z]$ and  $S={\sf k}_{DP}[X,Y,Z]$.

   \subsection{The case $T=(1,3,s^k,3,1), 3\le s\le 6.$}
   
  We start by analyzing the codimension three AG algebras having least values of the invariant  
  $
  \tau=\min\{i\mid T(A)_i=s\},
  $ where $s$ is the Sperner number of the AG algebra $A$.
  The corresponding Hilbert functions are given in tables \eqref{table1} and \eqref{table2}.
  
   \begin{theorem}
   \label{SLthm} 
   If $A\in \Gor(T), T=(1,3,s^k,3,1)$ with $3\le s\le 6$ and $\cha {\sf k}=0$ then $A$ is SL. When $\sf k$ is an infinite field of characteristic greater than $j=k+3$, then $A$ is almost SL.
   \end{theorem}
   
\begin{proof}  Assume $T=(1,3,s^k,3,1_j), 3\le s\le 6$ and let $F$ be a Macaulay dual generator for $A$ (see Definition \ref{def:F}).

The map $\ell^{j-2}: A_1\to A_{j-1}$ is an isomorphism for a general enough linear form $\ell\in A_1$, provided the Hessian determinant $\det\Hess(F)$ is not identically zero. This is ensured when $\cha {\sf k}=0$ by the Gordan-Noether theorem (Theorem \ref{GNlem}).

When $k=1$ there is nothing more to show. When $k=2$ the weak Lefschetz property for $A$ follows from  Proposition \ref{WLlem}(ii),(iii) and implies that the map $\ell:A_2\to A_3$ is an isomorphism.

Now, when $k\ge 3$, by Theorem \ref{pithm} $F$ has a (unique) tight annihilating scheme $\Z$  defined by a saturated ideal $J$ (see Definition \ref{def:tightann}), and $I=\Ann(F)$ satisfies  $I_t=J_t, t\in [2,j-2]$ by Theorem \ref{pithm}. Since $R/J$ is saturated, a general linear form $\ell\in R_1$ is a non zero-divisor on $R/J$, so the map $\ell^{j-2i}:A_i\to A_{j-i}$ is an isomorphism for $2\le i\le j/2$ (Lemma \ref{tightlemma}). The Theorem follows.
\end{proof}

\subsection{AG algebras of order two.}
Recall that the order $\nu(A)$ of an Artinian algebra $A$ is $\nu(A)=\min\{i\mid \dim_{\sf k} A_i \not=\dim_{\sf k} R_i\}$. In this section we treat the case of codimension three AG algebras of order two, that is, AG algebras $A$ with $T(A)_1=3$ and $T(A)_2\leq 5$.  We do not restrict $T(A)$ otherwise.

\begin{definition}
We say that an Artinian Gorenstein algebra $B$ has {\em almost constant Hilbert function} if $T(B)=(1, s^k, 1)$ for some integers $s,k>0$.
\end{definition}
\begin{lemma}\label{acSLlem} 
Suppose that an AG algebra $B$ has almost constant Hilbert function $T$ with Sperner number $s\leq 3$ or the multiplicity $k\ge s+1$. Assume additionally that $\cha {\sf k}=0$ or  $\sf k$ is infinite with $ \cha {\sf k}=p>j_B=k+1$. Then $B$ is strong Lefschetz.
\end{lemma}
\begin{proof}\label{AlmCSL} For codimension $s=1$ or $2$ this is well known: the latter is due to J. Brian\c{c}on  (see \cite[Proposition 2.15]{IMM}). When $s=3$ and $k=1$ there is nothing to show. When $s=3$ and $k=2$ the SL property follows from the WL property, shown in Proposition \ref{WLlem}(iii)  (or for $\sf k$ algebraically closed is easily checked from the classification of nets of conics \cite{AEI}).

Now assume $s=3$ and $k\geq 3$ and set $B=R/I$. Theorem \ref{pithm}  guarantees the existence of an ideal $J\subset I$ so that $R/J$ is saturated, defining a scheme $\Z\subset \mathbb P^2$ and $B_i=(R/J)_i$ for $i\leq s$. Since $R/J$ is saturated a general linear form $\ell$ is a non zero-divisor on $R/J$ and consequently the pair $(B,\ell)$ is strong Lefschetz by Lemma \ref{tightlemma}.

Now assume $s$ is arbitrary and $k\ge s+1$. Then, since Macaulay growth of $T_s=s$ to $T_{s+1}=s$ is maximum, by the Gotzmann theorem $J=(I_{\le s+1})$ is the saturated ideal defining a length $s$ punctual scheme $\Z\subset \mathbb P^{s-1}$ \cite[Proposition C.32]{IK}. The SL property follows from Lemma \ref{tightlemma}.
 \end{proof}

The following result is a crucial ingredient to our proofs in this section. 

\begin{proposition}
\label{prop:almostconstant}
Let $R=\k[x_1,\ldots, x_r]$ with $n\le 3$, let $A=R/I$ be an Artinian Gorenstein algebra  and assume $\cha \k=0$ or $\sf k$ is infinite with $\cha \k=p>j$, the socle degree of $A$. Assume that there is a linear form $x\in R_1$ such that the ring $B=R/(I:(x))$ has almost constant Hilbert function and 
the ring $A/(x)$ is SL. Then $A$ is almost SL.

 If $\cha \k=0$ then $A$ is SL.
\end{proposition}

\begin{proof}
Set $C=A/(x)=R/(I+(x))=k[y,z]/\overline{I}$, where $\overline{I}=I+(x)/(x)$ and set $j$ to be the socle degree of $A$. By Remark \ref{rem:B}, $B$ is AG with Macaulay dual polynomial $G=x\circ F$ and socle degree $j-1$. 
The following short exact sequence relates the three rings
\begin{equation}
\label{ses}
0 \to B(-1) \xrightarrow{\cdot x} A\to C\to 0
\end{equation}
Let $\ell\in R_1$ be a linear form. Multiplication by powers of $\ell$ in \eqref{ses} produce the following commutative diagram in which $K', K,K''$ denote the kernels and $L',L,L''$ denote the cokernels of the vertical maps

\begin{equation*}
\label{eq:3}
\small
 \begin{tikzcd}
& K' \arrow{d} \arrow{r}  &K   \arrow{d} \arrow{r} & K'' \arrow{d} &  \\
0 \arrow{r}{} & B_{i-1}  \arrow{d}{\ell'^{j-2i}} \arrow{r}  &A_i   \arrow{d}{\ell^{j-2i}} \arrow{r} & C_i \arrow{d}{\ell'^{j-2i}} \arrow{r} &0 \\
0 \arrow{r}{}  & B_{j-i-1} \arrow{r}  \arrow{d} &A_{j-i} \arrow{r}  \arrow{d}&  C_{j-i}  \arrow{r}  \arrow{d}&0 \\
& L'   \arrow{r}  &L    \arrow{r} & L'' &  
\end{tikzcd}
\end{equation*}
Since $T(B)$ is almost constant and $B$ has codimension at most three, hence also Sperner number at most three, $B$ is SL by Lemma \ref{acSLlem}. Consequently, the leftmost vertical maps in the above diagram are isomorphisms for $i>1$. Thus for $i\geq 2$ we have $K'=L'=0$ and the snake lemma yields that the middle vertical map has maximum rank if and only if the rightmost vertical  map does. However,  $C$ is assumed to be a strong Lefschetz algebra and so we see that multiplication by $\ell^{j-2i}$ has maximum rank on $C$ for  general $\ell$ and hence also on $A$ for $i\geq 2$. This shows that $A$ is almost SL.

It remains to show that when $\cha {\sf k}=0$ that for $i=1$, multiplication by $\ell^{j-2}:A_1\to A_{j-1}$ is an isomorphism. Setting $I=\Ann(F)$, the Gordan-Noether Theorem \ref{thm:hessian} completes the proof that $A$ is SL when $\cha {\sf k}=0$.
\end{proof}

\begin{theorem}
\label{thm1,3,4}
 Let $A$ be an AG $\k$-algebra of socle degree $j$  with  $H(A)=(1,3,4, \ldots)$. If  $\ch{\sf k}=0$ then $A$ is SL. If $\ch {\sf k}=p>j$ then $A$ is almost strong Lefschetz.
\end{theorem}
\begin{proof}
From the characterization of codimension three Gorenstein sequences in Corollary \ref{charcor} we see that for $T(A)=(1,3,4,\ldots)$ we have $T(A)_3=3$ if and only if $j=4$, otherwise $T(A)_3\in \{4,5\}$. The case $T(A)_3\in \{3,4\}$ is treated in Theorem \ref{SLthm} (Note that $T(A)_2=T(A)_3=4$ implies by Lemma \ref{cod3SIem}
that $T=(1,3,4^k,3,1)$ ).

Now assume $T(A)_3=5$. The precise form of the Hilbert function in this case is $T=(1,3,4,5,\ldots,s-1,s^k,s-1,\ldots ,4,3,1)$ according to Corollary \ref{charcor}. Suppose $A={\sf k}[x,y,z]/I$. 
In this case there must be a linear relation among the two basis elements of $I_2$, else $T(A)=(1,3,4,4,\ldots)$. So
 we have $I_2\cong \langle xy, xz\rangle$ or $I_2\cong\langle x^2,xy\rangle$, up to $\Pgl(3)$ action. Then either  $(y,z)\subseteq I:(x)$ or $(x,y)\subseteq I:(x)$, respectively and thus the AG ring $B={\sf k}[x,y,z]/(I:(x))$ has constant Hilbert function $T(B)=(1^{j})$. Moreover $A/(x)$ is a codimension two algebra and hence is SL under the given assumption on characteristic. Proposition \ref{prop:almostconstant} now yields that $A$ is SL if $\cha {\sf k}=0$ and is almost SL if $\cha{\sf k}=p>j$. 
\end{proof}

\begin{remark}
An alternate proof of the strong Lefschetz property in the case when the defining ideal of $A$ contains the forms $xy, xz$ is to observe that this implies $A$ can be written as a connected sum $A=U\#_\k V$ where $U$ is an AG quotient of $\k[x]$ and $V$ is an AG quotient of $\k[y,z]$. Since such $U,V$ are SL under the assumptions on the characteristic of $\k$ in the statement of Theorem \ref{thm1,3,4}, it follows from \cite[Proposition 5.7]{IMS} or \cite[Theorem 3.76, Proposition 3.77ii]{H-W} that $A$ is SL as well. This alternative proof is fine for $\cha {\sf k}=p>j$, also.
\end{remark}

\subsubsection{Case $T=(1,3,5,6^k,5,3,1)$.}
\begin{theorem}
\label{thm1,3,5,6}
 Let $A$ be an Artinian Gorenstein algebra of socle degree $j\geq 5$ with $T(A)=(1,3,5, 6^{j-5}, 5, 3, 1)$. If $\ch {\k}=0$, then $A$ is SL. If $\ch{\k}=p>j$, then $A$ is almost SL.
 \end{theorem}
\begin{proof}
Theorem 3.3 deals with $j=5$, so we may assume $j\ge 6$.
Let $A=R/I$ with $R=\k[x,y,z]$. Then the lowest graded components of $I$ are given by $I_2=\langle f \rangle$ for some $f\in R_2$ and $I_3=\langle xf, yf, zf, g \rangle$  for some $g\in R_3\setminus (x,y,z)f$. 

{\em Case (i):} If $f,g$ do not form a regular sequence then they have a common linear divisor. Without loss of generality assume $x$ is this common divisor and thus $f=xf'$ and $g=xg'$ with  $f'\in R_1, g'\in R_2$. Set $C=A/(x)$ and $B=A/(0:_A x)$ and note that both $B$ and $C$ are AG algebras of codimension two; the latter is because $(0:_Ax)$ contains the linear form $f'$. By Remark \ref{rem:B}, $B$ is AG of socle degree $j-1$.

Since $f',g' \in (0:_Ax)$, we have  $T(B)_1=T(B)_2=2$ and therefore by the properties of O-sequences and the symmetry of AG Hilbert functions one concludes that
\[
T(B)_i=
\begin{cases}
1 & \text{for } i=0,  i= j-1\\
2 & \text{for } i\in [1, j-2]\\
0 & \text{for } i\geq j.
\end{cases}
\]
Thus $B$ has almost constant Hilbert function and $C$ is SL since it is a codimension two Artinian algebras, so Proposition \ref{prop:almostconstant} applies to conclude that $A$ is SL, as desired when $\cha {\sf k}=0$, or that $A$ is almost SL when $\cha {\sf k}=p>j$.\vskip 0.2cm

{\em Case (ii):} Suppose $f,g$  form a regular sequence. Then $A$ is a complete intersection with  defining ideal $I=(f,g,h)$ so that $\deg(h)=j-2\geq 4$. Indeed, the Hilbert function of $R/(f,g)$ being $1,3,5, 6,6, \cdots$, $I$ must have  a unique minimal generator in degree $j-2$. To see that $I$ has no  minimal generators in degree $j-1$ or $j$ consider the symmetry of the graded free resolution of $A$. In  detail, the existence of minimal generators of $I$ in degree $j-1$ or $j$ would imply by symmetry the existence of minimal syzygies of $I$ in degree $4$ or $3$ respectively. As the unique minimal syzygy of $(f,g)$ is in degree 5, and  $\deg(h)\geq 4$, this yields a contradiction.

A general linear form $\ell\in R_1$ satisfies $\ell^j\neq 0$ and by Gordan-Noether, when $\cha {\sf k}=0$ it yields a bijection $\ell^{j-2}:A_1\to A_{j-1}$. Set $A'=\k[x,y,z]/(f,g)$.  Since ${\rm depth}(A')=1$, a general linear form $\ell$ is  a non zero-divisor on $A'$.  Observe that $A'_i=A_i$ for $i\leq j-3$ and thus the maps $\ell^{j-2i}:A_i\to A_{j-i}$ are bijective for $i\geq 3$ since $\ell$ is a non zero-divisor on $A'$. It remains to confirm that the map $\ell^{j-4}:A_2\to A_{j-2}$ is also bijective. 

The following argument is inspired by \cite[proof of Proposition 2.4]{BFP}. 
Assume by way of contradiction that for every $L\in A_1$ there exists $q\in A_2$ so that $L^{j-4}q=0$. 

{\em Claim:} for each  $L\in A_1$ that is regular (non zero-divisor) on $A'$ there exists exactly one $q\in \PP(A_2)$ so that $L^{j-4}q=0$. Indeed, suppose that $L^{j-4}q=L^{j-4}q'=0$ in $A=A'/(h)$. This yields $L^{j-4}q\in (h)A'$ and $L^{j-4}q'\in (h)A$, which, since $h\in A^\prime_{j-2}$, gives relations $L^{j-4}q=\lambda h$ and $L^{j-4}q'=\mu h$ in $A'$ for some $\lambda, \mu \in \k^\times$. The identity $L^{j-4}(\mu q-\lambda q')=0$ in $A'$ combined with the fact that $L$ is regular on $A'$ yields that $q, q'$ are linearly dependent.

Fix $v\in R_1$, set $L(t)=\ell+tv$ and set $C=\{L(t): t\in \k\}$ to be the image of the regular map $L: \AA_\k^1\to \PP(A_1), t\mapsto \ell+tv$ . The set of forms that are non zero-divisors on $A'$ forms a nonempty Zariski open set.  Denote by $U'$ the intersection of the set of non zero-divisors on $A'$ with $C$. Then $U'$ is a nonempty (it contains $L(0)=\ell$) Zariski open subset of $C$. Moreover the set $U=L^{-1}(U')$ is a nonempty Zariski open subset of $\AA_k^1$ which contains $0$. For each $L(t)\in U'$ the claim yields a unique $q(t)\in \PP(A_2)$ so that $L(t)^{j-4}q(t)=0$.

Consider more generally the incidence correspondence $\Gamma\subset \PP(A_1)\times \PP(A_2)$ where 
\[
\Gamma=\{(L(t),q) \mid L(t)=\ell+tv\in A_1, q\in A_2, L(t)^{j-4}q=0\}
\]
and define  $\pi_2: \Gamma\to \PP(A_2), \pi_2(L(t),q)=q$. Let $U'=\pi_1(U)$. Based on the claim, there is a well defined regular map
\[
\phi: U\to \PP(A_2), \ \phi(t)=\pi_2(L(t))=q(t)=(u_0(t): \ldots : u_4(t))
\]
which induces a homomorphism
\[
\psi: \k[\PP(A_2)] \to \OO_{\AA^1}(U), \text{ given by }\theta\mapsto \theta(q(t) ),\ \forall \theta\in \k[\PP(A_2)]=\k[x_0,\ldots, x_4].
\]
Composing $\psi$ with the inclusions $\OO_{\AA^1}(U)\hookrightarrow \OO_{0,\AA^1}=\k[t]_{(t)} \hookrightarrow \widehat{\OO_{0,\AA^1}}=\k[[t]]$ and taking $\theta_i=x_i$ to be the coordinate functions on $\PP(A_2)$ allows to express each coordinate $u_i(t)$ of $q(t)$ as a power series in $\k[[t]]$.
Hence $q(t)$ itself can be written as
\[
q(t)=\sum_{i=0}^\infty q_i t^i   \text{ with } q_i\in A_2.
\]
Rewrite $L(t)^{j-4}q(t)=0$ as
\[
(\ell+tv)^{j-4}\left ( \sum_{i=0}^\infty q_i t^i\right ) =0, \forall v\in A_1
\]
where $0\neq q_0=q(0)\in R_2$ is such that $\ell^{j-4}q_0=0$.
The coefficient of $t$ in the above displayed equation is $(j-4)\ell^{j-5}q_0v+\ell^{j-4}q_1$, which must be 0 for all $v\in A_1$. Since 
\[
\ell^{j-5}\left( (j-4)q_0v +\ell q_1 \right)=0,  \forall v\in A_1
\]
we see that the kernel $K$ of the map  $\ell^{j-5}:A_3\to A_{j-2}$ has dimension 
\[
\dim_\k K\geq \dim_\k  {\rm Span}\{(j-4)q_0v +\ell q_1\mid v\in A_1\} \geq \dim_\k  \{q_0v \mid v\in A_1\}.
\]
For the last displayed inequality above we have used that $j-4\neq0$ since $\cha \k>j$.

As $A_{\leq j-3}=A'_{\leq j-3}$, we have that $\ell^{j-5}:A_2\to A_{j-3}$ is injective and thus, by Gorenstein duality, the map $\ell^{j-5}:A_3\to A_{j-2}$ is surjective for a general $\ell\in R_1$. This means $\dim_\k K=1$ which yields $ \dim_\k  \{q_0v \mid v\in A_1\}=1$ (note that the latter cannot be 0 since $q_0\ell\neq 0$). Consequently, we have $\dim((q_0)\cap(f,g))_3=2$. Since $(f,g)_3=\langle xf, yf, zf, g \rangle$, it follows that there exists a nonzero element in $((q_0)\cap(f))_3$ and thus $\gcd(q_0,f)\neq 1$ (otherwise $(q_0)\cap(f)=(q_0f)$ is generated in degree 4).

Lastly, recall that $\ell^{j-4}q_0=0$ in $A$, but $\ell^{j-4}q_0\neq0$ in $A'$ so $\ell^{j-4}q_0\in (f,g,h)\setminus(f,g)$. Since $\deg(h)=\deg(\ell^{j-4}q_0)$ we see that $I=(f,g,h)=(f,g, \ell^{j-4}q_0)$. It was shown above that $I$ must be generated by a regular sequence, but the sequence $f,g, \ell^{j-4}q_0$ is not a regular sequence since $\gcd(q_0,f)\neq 1$, a contradiction. This completes the proof.
\end{proof}

\subsection{Residue fields of positive characteristic}\label{charpsec}
In the cases covered by our Main Theorem, our proofs show that when $\sf k$ is an infinite field of $\cha{\sf k}=p>j$, where $j$ is the socle degree of $A$, then $A$ is almost strong Lefschetz: that is, the maps $\ell^{j-2i}:A_i\to A_{j-i}$ are isomorphisms for $2\le i\le j/2$, when $\ell$ is a general enough linear form. Thus, 
$\cha {\sf k}=0 $ is needed only for the Gordan-Noether theorem, that we use to show 
$\ell^{j-2}: A_1\to A_{j-1}$ is an isomorphism. We propose

\begin{conjecture} 
\label{conj}
When $A$ is an AG algebra of codimension three, and $\sf k$ is an infinite field of $\cha {\sf k}=p>j_A$, then the map $\ell^{j-2}:A_1\to A_{j-1}$ is an isomorphism for a general enough linear form $\ell$.
\end{conjecture}

 In the following case we are able to bypass the Gordan-Noether theorem and give a strengthening of Lemma~\ref{acSLlem} for codimension three AG algebras of almost constant Hilbert function.
 
\begin{lemma}\label{13k1SLlem} Let $A$ be a graded AG algebra over an algebraically closed field $\k$ of characteristic $p\not= 2$ or $3$. Let $A=R/I, I=\Ann(F)$ and assume $T(A)=(1,3^k,1), k\ge 2$. Then $A$ is strong Lefschetz.
\end{lemma}
\begin{proof} When $k\ge 3$ (so $j\ge 4$) Theorem \ref{pithm} shows that $R/J, J=(I_{\le 3})$ defines a tight annihilating scheme $\Z$ of $F$; this suffices to show that $A$ is SL as in the proof of Lemma \ref{acSLlem}. 

For $T=(1,3,3,1)$, in the non-CI case $I_2$ determines an algebra $A'=R/(I_2)$ of  Hilbert function $H=(1,3,3,\underline{3})$, by the structure theorem for codimension three AG algebras, so $A'$  has a tight annihilating scheme, which suffices to show $A$ is strong Lefschetz (Lemma \ref{tightlemma}). When $H=(1,3,3,1)$ and $I=(I_2)$ is a complete intersection, the classification of nets of conics
in \cite[Table 1]{AEI}, which is valid for  $\cha {\sf k}=p$ (ibid. p. 6, 80)
shows that the dual generator $F\in S_3$ is one of those listed in \#8b,c, \#7c, or \#6d of 
Table~1 there. It is straightforward to check that in each case $A=R/\Ann(f)$ is SL.\footnote{\#8c,\#7c in \cite[Table 1]{AEI} specialize to \#6d where $I=(x^2,y^2,z^2)$ which is evidently SL; and SL in the smooth case \#8b is immediate.} This completes the proof.
\end{proof}
We now give an example to show that the condition $\cha {\sf k}>j$ is necessary for our results about almost strong Lefschetz AG algebras.

\begin{example}[Counterexamples: AG algebras, $\cha  {\sf k}< j$]
\label{exp}
\begin{enumerate}[(a)]
\item The following example shows that the hypothesis $\cha {\sf k}=p>j$ is necessary in our Conjecture \ref{conj}. Take $p=3$ and let $I=(x^3,y^3,z^2)$ define an algebra with Hilbert function $T(A)=(1,3,5,5,3,1)$ and socle degree $j=5>p$. Then multiplication by $(x+y+z)^3$ is not full rank from $A_1$ to $A_4$.  Since $I$ is a monomial ideal, there exists a change of coordinates on $A$ which takes any general linear form to $x+y+z$.
\item The following example shows that codimension two AG algebras may not be almost SL when $\cha {\sf k}=p\le j$. Take $p=2$ and let  $I=(x^4,y^4)$ define an algebra with Hilbert function $T(A)=(1,2,3,4,3,2,1)$. Then $ (x+y)^2$ is not full rank from $A_2$ to $A_4$, so arguing that $x+y$ is general as in part (a) yields that $A$ is not almost SL.
\item The following example shows that a codimension three AG algebra may not be almost SL when $\cha {\sf k}=p\le j$. Take $p=3$ and let $I=(x^3,y^3,z^4)$ define an algebra with Hilbert function $T(A)=(1,3,6,8,8,6,3,1)$. Then $(x+y+z)^3$ is not full rank from $A_2$ to $A_5$, so arguing that $x+y+z$ is general as in part (a) yields that $A$ is not almost SL.
\item  In the  previous examples, also $\ell^j: A_0\to A_j$ is the 0 map. However, taking a complete intersection $I=(x^3,y^3,z^{14})$, of socle degree 17, and  $\cha {\sf k}=13$ we will have ${\ell=x+y+z}$ satisfying $\ell^{13}: A_2\to A_{15}$ is not of full rank ($z^2$ is in the kernel), but $\ell^{17}: A_0 \to A_{17}$ has the non-zero image $\binom{17}{2,2,13}x^2y^2z^{13}$. A check with Macaulay2 shows $\ell^{2i+1}:A_{8-i}\to A_{9+i}$ does not have full rank for $i\in \{5,6,7\}$, so $A$ is not almost strong Lefschetz. 
\end{enumerate}
\end{example}

There has been some substantial study of the Lefschetz properties and Jordan types of monomial ideals, in particular monomial complete intersections, in low characteristic, see, for example \cite{BrKa,LuNi}. Some overview is presented in \cite[Section 3.3]{IMM}, discussing the Clebsch-Gordan theorem in low characteristics.

Some have conjectured that all codimension three graded AG algebras over fields of characteristic ${\sf k}=0$ should be weak Lefschetz (as \cite[\S 1]{BMMNZ}), a somewhat ambitious conjecture itself, given what we know. We expect that not all AG codimension three algebras would be SL, even in characteristic zero. Given the substantial connection of the failure of WL in different contexts with geometry, for example \cite{HSS,BMMN}, and the fact that an open dense subset in each $\Gor(T)$ for codimension three is SL (Proposition \ref{gencompprop}), we could expect a strong geometric flavor to finding a counterexample of a codimension three non-SL AG algebra, for $\cha {\sf k}=0$ or $\cha{\sf k}=p>j$.

\begin{ack} We would like to thank Leila Khatami who contributed insight in many of our early discussions of these issues. We thank Chris McDaniel for helpful comments. We are grateful to the series of annual Lefschetz Properties In Algebra and Combinatorics conferences and workshops, some of which we each participated in, at Mittag Leffler (July 10-14, 2017),  CIRM Levico (June 25-29, 2018), CIRM Luminy (October 14-18, 2019), and MFO Oberwolfach (September 27-October 4, 2020). The fourth author is supported by NSF grant DMS--2101225.
 \end{ack}

\end{document}